\documentclass[a4paper, reqno]{amsart}

\usepackage[british]{babel}
\usepackage[T1]{fontenc}
\usepackage[ansinew]{inputenc}
\usepackage{amssymb} 
\usepackage{amsmath} 
\usepackage{upref} 
\usepackage{url}
\usepackage{paralist}
\usepackage{bm}
\usepackage{color}
\usepackage{ifthen}

\usepackage{hyperref}

\hypersetup{pdfauthor={Matthias Lenz},pdftitle={Matroids and log-concavity}, 
  pdfsubject={Matroid theory}, pdfkeywords={matroid, Mason's conjecture, log-concave sequence, 
            unimodal sequence, f-vector, h-vector, Tutte polynomial, zonotopal algebra %
            } }

\newtheorem{Theorem}{Theorem}[section]
\newtheorem{Corollary}[Theorem]{Corollary}
\newtheorem{Proposition}[Theorem]{Proposition}
\newtheorem{Lemma}[Theorem]{Lemma}

\theoremstyle{definition}
\newtheorem{Definition}[Theorem]{Definition}
\newtheorem{Example}[Theorem]{Example}

\theoremstyle{remark}
\newtheorem{Remark}[Theorem]{Remark}

\newcommand{\abs}[1]{\left|#1\right|}
\newcommand{\rank}{\mathop{\mathrm{rk}}}

\newcommand{\st}{s.\,t.\ } %
\newcommand{\ie}{\textit{i.\,e.\ }} %
\newcommand{\eg}{\textit{e.\,g.\ }} %

\newcommand{\K}{\mathbb{K}}

\newcommand{\Pcal}{\mathcal{P}}

\newcommand{\spa}{\mathop{\mathrm{span}}}

\newcommand{\hilb}{\mathop{\mathrm{Hilb}}}

\numberwithin{equation}{section}

\begin{document}

\date{\today}
\title%
{%
 Matroids and log-concavity
}

\author%
{%
Matthias Lenz%
 }
\address{Merton College,
OX1 4JD, Oxford, United Kingdom
}
 
 \thanks{This work was carried out while the author was at TU Berlin. He was supported by
 a Sofia Kovalevskaya Research Prize of Alexander von Humboldt Foundation awarded to Olga Holtz.}

\email{lenz@maths.ox.ac.uk}

\begin{abstract}
 We show that $f$-vectors of matroid complexes of realisable matroids are %
 log-concave. This was conjectured
 by Mason in 1972.   
 Our proof uses the recent result by Huh and Katz who showed that the 
 coefficients of the characteristic polynomial of a realisable matroid form a log-concave sequence.
We also discuss the relationship between log-concavity
 of $f$-vectors and $h$-vectors of matroids.
 In the last section we 
 explain the connection between zonotopal algebra and  $f$-vectors and characteristic polynomials of matroids.
\end{abstract}

\subjclass[2010]{%
Primary: 
05A20, %
05B35, %
 Secondary:
 05E45, %
 05C31.%
}

\keywords{matroid, Mason's conjecture, log-concave sequence, unimodal sequence, $f$-vector, $h$-vector, 
Tutte polynomial, zonotopal algebra.}
       
\maketitle

\section{Introduction}
Let $M=(E,\Delta)$ be a matroid of rank $r$. $E$ denotes the ground set and $\Delta \subseteq 2^E$
 denotes the matroid complex, \ie the abstract simplicial complex of independent sets. 
Let $f=(f_0,\ldots, f_r)$ be the $f$-vector of $\Delta$, \ie $f_i$ is the number of sets of cardinality $i$
 in $\Delta$. 
 Dominic Welsh conjectured in 1969 \cite{welsh-1971} that the $f$-vector of a matroid complex is \emph{unimodal}, \ie
  there exists  $j\in\{0,1,\ldots, r\}$ \st
 $f_0 \le f_1 \le \ldots   \le f_j \ge \ldots \ge f_r$. 
Three successive strengthenings of this conjecture were proposed by John Mason in 1972 \cite{mason-1972}.
 The weakest of them is \emph{log-concavity} of the $f$-vector, \ie
\begin{align}
   f_i^2 \ge f_{i-1}f_{i+1} \text{ for } i=1,\ldots, r - 1. \label{eq:logconcavity}
\end{align}
Since then, these conjectures have received considerable attention. See for example
   \cite{brown-colbourn-1994, cheng-masuyama-2010, dawson-1984,dowling-1980, hamidoune-salaun-1989, 
 johnson-kontoyiannis-madiman-2011,kahn-neiman-2011, mahoney-1985, seymour-1975, 
 stanley-1981,wagner-2008,zhao-1985}.
  Carolyn Mahoney proved log-concavity for cycle matroids of outerplanar graphs in 1985 \cite{mahoney-1985}.
David Wagner \cite{wagner-2008} describes further partial results, several stronger variants of Mason's conjecture, 
 and other sequences of integers that are associated with a matroid and that are conjectured to be log-concave.   
Log-concave sequences arising in combinatorics have been studied by many authors. For an overview, see the surveys by Francesco Brenti and Richard Stanley \cite{brenti-1994,stanley-1989}.

Recall that a matroid is \emph{realisable}
 if it is equivalent to a matroid 
 whose ground set is a list of vectors 
 in a vector space over some field $\K$ 
 and whose independent sets are the linearly independent subsets
 of this list. 
Our main result is the following theorem.
\begin{Theorem}
\label{Theorem:fvectorLogConcave}
 The $f$-vector of the matroid complex of a realisable matroid is %
 log-concave.
\end{Theorem}
This theorem follows from the log-concavity  
of the characteristic polynomial of a realisable matroid that was shown by June Huh and Eric Katz 
 and a connection between $f$-vectors and the characteristic polynomial that was first discovered by Tom Brylawski.

 The strongest of Mason's three conjectures \cite{mason-1972} %
 is ultra-log-concavity, \ie the conjecture that the following 
 inequalities hold:
\begin{equation}
 \frac{f_i^2}{\binom{f_1}{i}^2} \ge  \frac{f_{i-1}}{\binom{f_1}{i-1}}\frac{f_{i+1}}{\binom{f_1}{i+1}} 
 \text{ for } i=1,\ldots, r-1.
\end{equation}
This conjecture was one of the main topics of a workshop at AIM in 2011\footnote{Workshop on 
 \emph{Stability, hyperbolicity, and zero localization of functions}, 
 December 5 to December 9, 2011 
 at the 
American Institute of Mathematics, Palo Alto, California. Organised by 
 Petter Br\"and\'en, George Csordas, Olga Holtz, and Mikhail Tyaglov.\\
\url{http://www.aimath.org/ARCC/workshops/hyperbolicpoly.html}
}.

\smallskip
Finding inequalities satisfied by $f$-vectors of matroid complexes is interesting because it 
is a step towards the classification
of $f$-vectors and $h$-vectors of matroid complexes.
In this context, 
it is also interesting to know that the
convex hull of the set of $f$-vectors of matroid complexes on $N$
 elements is a simplex whose vertices are $f$-vectors of uniform matroids~\cite{kozlov-1997}.
Johnson, Kontoyiannis, and Madiman \cite{johnson-kontoyiannis-madiman-2011} show that 
a stronger version of Theorem~\ref{Theorem:fvectorLogConcave}
would 
imply a bound on the entropy of the cardinality of a random independent set in a matroid.
Our log-concavity results might also help to prove statements about coefficients and zeroes of various graph polynomials. 
The 
chromatic polynomial, the nowhere-zero flow polynomial, the 
 critical configuration polynomial, the shelling polynomial, and the reliability polynomial 
 are all related to the matroid polynomials studied in this article. 
  The connection can be made via the Tutte polynomial. See \cite{ellis-merino-2011} for  details.
Brown and Colbourn state  that our log-concavity results might have applications to the theory of 
network reliability \cite[p. 117]{brown-colbourn-1994}.
\subsection*{Organisation of the article}
In Section~\ref{Section:MatroidAndMatroidPolynomials} we introduce the $f$-polynomial and the 
characteristic polynomial of a matroid.
Recently,   Huh and  Katz  %
proved that the characteristic polynomial of a 
 realisable matroid is log-concave (a univariate polynomial is log-concave if its coefficients form a log-concave sequence).
 In Section~\ref{Section:Extensions} we establish a connection between the characteristic polynomial 
 and the $f$-polynomial.
In conjunction with the result by Katz and Huh,
 this implies log-concavity of the $f$-polynomial of realisable matroids.
 In Section~\ref{Section:Supplementary} we  discuss connections between (strict) log-concavity 
 of $h$-vectors and $f$-vectors and the matroid operation thickening.
In Section~\ref{Section:ZonotopalAlgebra} we give a brief introduction to zonotopal algebra and explain 
 how the $f$-polynomial 
 and the characteristic polynomial are related to it. 
 Zonotopal algebra is the theory of several classes of vector spaces of polynomials that can be 
 associated with a realisation of a matroid. 
 The Hilbert series of these spaces are matroid invariants.

\section{Matroid polynomials}
\label{Section:MatroidAndMatroidPolynomials}
In this section we review the definitions of some matroid polynomials. 
We assume that the reader is familiar with matroid theory. 
A good reference  is the book by James Oxley \cite{MatroidTheory-Oxley}.

Recall that we denote by $M=(E,\Delta)$ a matroid of rank $r$. 
Let $\rank$ denote the rank function of $M$.
The \emph{Tutte polynomial} \cite{brylawski-oxley-1992} of  $M$ is 
 defined as %
\begin{align}
T_M(x,y) &= \sum_{A\subseteq E}(x-1)^{r-\rank(A)} (y-1)^{\abs{A}- \rank(A)}. %
\end{align}
An important specialisation of the Tutte polynomial is the \emph{characteristic polynomial}
\begin{align}
\chi_M(q) &=  (-1)^{r}T_M(1-q,0) = \sum_{A\subseteq E} (-1)^{\abs{A}} q^{r - \rank(A)}.  
\end{align}
The \emph{reduced characteristic polynomial} is defined as
\begin{align}
 \bar\chi_M(q) = \frac{1}{q-1}\chi_M(q).
\end{align}
Note that since $E\neq \emptyset$, $\chi_M(q)$ vanishes for $q=1$, so $\bar\chi_M(q)$ is indeed a polynomial.
Huh and Katz proved the following theorem, extending an earlier theorem by Huh~\cite{huh-2012}.
\begin{Theorem}[\cite{huh-katz-2012}]
\label{Theorem:KatzHuh}
If $M$ is a realisable matroid, then the coefficients of its reduced characteristic polynomial $\bar\chi_M(q)$ form a log-concave
 sequence. 
\end{Theorem}

It is easy to see that log-concavity of $\bar\chi_M(q)$ implies log-concavity of $\chi_M(q)$.
We are interested in the \emph{$f$-polynomial} of the matroid given by
\begin{align}
f_M(q) = T_M(1+q, 1) &= %
  \sum_{%
   A \in \Delta } q^{r - \rank(A)} = \sum_{i=0}^r f_i q^{r-i}.
\end{align}

\section{Free (Co-)Extensions%
}
In this section we introduce free (co-)extensions of matroids.
 This helps us to establish a connection between the characteristic polynomial 
 and the $f$-polynomial. In conjunction with Theorem~\ref{Theorem:KatzHuh},
 this connection implies log-concavity of the $f$-polynomial of realisable matroids.
\label{Section:Extensions}
\begin{Definition}
 Let $M=(E,\Delta)$ be a matroid of rank $r$ and let $e\not \in E$. 
 The \emph{free extension} of $M$ (by $e$) is the matroid  $M + e=(E\cup \{e\}, \Delta + e)$, where
 \begin{align}
    \Delta + e :=  \Delta \cup \{ (I\cup \{e\}) : I \in \Delta \text{ and } \abs{I}\le r-1 \} .
 \end{align}
 \end{Definition}
 Several properties of the free extension are described in \cite[7.3.3.~Proposition]{brylawski-1986}.
\begin{Remark}
 \label{Remark:ExtensionRealization}
 If $M$ is realised over the field $\K$ by the list of vectors $X \subseteq \K^r$, then $M + e$  is realised by 
 the list $(X, x)$, where $x\in \K^r$ is a vector
  that is not contained in any (linear) hyperplane spanned by the vectors in $X$. If $\K$ is a finite field, such a vector might not exist.
  However, if $M$ is realisable over the field $\K$, it is also realisable over the infinite field $\K(t)$ 
 of rational functions in $t$ with coefficients in~$\K$.
\end{Remark}
Recall that the \emph{dual matroid} $M^*=(E,\Delta^*)$ is given by 
\begin{align}
  \Delta^* = \{ A : 
   \rank(E\setminus A)=r
  \} .
\end{align}
The dual matroid has rank $ r^* = \abs{E}-r$ and its rank function is given by $\rank^*(A)= \abs{A} + \rank(E\setminus A) - r$.
The Tutte polynomial satisfies $T_M(x,y)=T_{M^*}(y,x)$.
We will use  the \emph{free coextension} $M \times e$ of a matroid $M$ which is defined as
\begin{align}
  M \times e := (M^* + e )^*.
\end{align}
Equivalently, the free coextension of $M$ is the extension by a non-loop $e$ which is contained in every
dependent flat \cite[Section 7.3]{MatroidTheory-Oxley}.

\begin{Proposition}
\label{Proposition:IndependenceCharacteristic}
Let $M$ be a matroid of rank $r$  and let $M \times e$ denote its free coextension. Then,
 \begin{align}
  (-1)^{r+1} \chi_{M \times  e}(-q) &= %
  (1 + q)f_M(q)    .
 \end{align}
\end{Proposition}
\begin{proof}
For the proof of this statement, we use the fact that both the characteristic polynomial and the $f$-polynomial are evaluations
 of the Tutte polynomial.
 Note that the matroid $M \times  e$ has rank $r+1$. 
 To simplify notation, the rank functions of $M^*$ and $M^* + e$ are both denoted by $\rank^*$.
\begin{align}
  (-1)^{r+1}\chi_{M \times  e} (-q) &= T_{M  \times  e}(1+q,0) = T_{M^* + e}(0,1+q) \\
  &= \sum_{A\subseteq E \cup\{ e \} } (-1)^{ r^* - \rank^*(A)} q^{\abs{A}-\rank^*(A)} \\
\begin{split}
  &= \sum_{A\subseteq E} \Bigl((-1)^{ r^* - \rank^*(A)} q^{\abs{A}-\rank^*(A)} 
      \\ &\qquad\qquad\qquad\quad 
      + (-1)^{ r^* - \rank^*(A \cup e)} q^{\abs{A}+1 - \rank^*(A\cup e)} \Bigl) 
      \label{eq:IndependenceCharacteristicA}
\end{split}
      \\
  &= (1+q)\sum_{\substack{A\subseteq E\\\rank^*(A)=r^*}}  q^{\abs{A} - r^*} \label{eq:IndependenceCharacteristic}
  = (1+q) T_{M^*}(1,1+q)  \\
 &=  (1+q) T_{M}(1+q,1 )  %
  = (1+q) f_{M}(q)
\end{align}
\eqref{eq:IndependenceCharacteristic} is equal to \eqref{eq:IndependenceCharacteristicA}
 because $\rank^*(A)<r^*$ implies $\rank^*(A \cup e)= \rank^*(A)+1$. For those $A$, the summands vanish.
\end{proof}

\begin{Remark}
Proposition~\ref{Proposition:IndependenceCharacteristic} has been discovered more than thirty years ago by Tom Brylawski.
It has appeared implicitly in \cite{brylawski-1977} and explicitly in
\cite[Remark 6.15.3c]{brylawski-1980}.
 It seems however that this result
  has not been widely known in the community. It was for example overlooked by Huh and Katz. The author
  rediscovered it independently.
  We will give another proof in Section~\ref{Section:ZonotopalAlgebra}.
\end{Remark}

\begin{proof}[Proof of Theorem~\ref{Theorem:fvectorLogConcave}]
Combine Proposition~\ref{Proposition:IndependenceCharacteristic} and Theorem \ref{Theorem:KatzHuh}.
Bear in mind that free coextensions of realisable matroids are realisable (cf.~Remark~\ref{Remark:ExtensionRealization}).
\end{proof}

\begin{Example}
 We consider the uniform matroid $U_{2,6}$, \ie the matroid on six elements where every set
 of cardinality at most two is independent. 
 Note that $U_{2,6}  \times  e = (U_{4,6} + e)^* = U_{4,7}^* = U_{3,7} $. 
\begin{align*}
  f_{U_{2,6}}(q) &=  q^2 + 6q + 15 \\
  (-1)^3\chi_{U_{3,7}}(-q) &=  q^3 + 7q^2 + 21q + 15  = (q+1)f_{U_{2,6}}(q) \\
\end{align*}
\end{Example}

\section{$h$-vectors, $f$-vectors, and strict log-concavity}
\label{Section:Supplementary}
This section contains some  results on connections between (strict) log-concavity of
$h$-vectors and $f$-vectors and the matroid operation thickening.
In Subsection~\ref{Subsection:hvectors} we show that 
 log-concavity of $h$-vectors implies strict log-concavity of $f$-vectors.
 In Subsection~\ref{Subsection:Thickenings} we  show that strict log-concavity
 of $f$-vectors implies strict log-concavity of $h$-vectors of certain thickenings 
 of a matroid. 
 In Subsection~\ref{Subsection:Modes} we  discuss possible locations of the modes of $f$-vectors.

As one might expect, a sequence of real numbers is called \emph{strictly log-concave} 
if it is log-concave and all inequalities are strict.
\subsection{$h$-vectors and strict log-concavity}
\label{Subsection:hvectors}
In this subsection we show that log-concavity of $h$-vectors implies
  strict log-concavity of $f$-vectors.
 The former was shown  recently by June Huh 
 in the case of matroids that are realisable over a field of characteristic zero
\cite{huh-hvector-2012}. 
The fact that %
$f$-vectors 
 of a large class of matroid complexes are strictly log-concave 
 indicates that they
 might satisfy even stronger inequalities as Mason conjectured.
\begin{Definition}
\label{Definition:hpolynomial}
Let $M$ be a matroid of rank $r$. Its \emph{$h$-vector} $(h_0,\ldots, h_r)$  consists of the 
 coefficients of the \emph{$h$-polynomial}
 defined by the equation $h_M(q) = \sum_{i=0}^r h_i q^{r-i} = f_M(q-1)$, \ie
\begin{align}
   h_j = \sum_{i = 0 }^j (-1)^{j-i} \binom{r- i}{j-i} f_i \quad \text{for } i=0,\ldots, r.
 \label{equation:hvector}
 \end{align}
\end{Definition}
It is well-known that 
log-concavity of $h$-vectors implies log-concavity of $f$-vectors
(see \cite[Corollary~8.4]{brenti-1994}, \cite[Proposition~6.13]{brylawski-1980}, \cite{dawson-1984}).
In fact, it implies even strict log-concavity of $f$-vectors.
This is a consequence of the following lemma.

\begin{Lemma}
 \label{Lemma:himpliesfLC}
Let $a_0,\ldots, a_r$ be non-negative integers and $a_0\neq 0$.
 Suppose that the polynomial  $a(q)=\sum_{i=0}^r a_i q^{r-i}$ is log-concave. Then, the
 polynomial $ b(q) = \sum_{i=0}^r b_i q^{r-i} = a(q+1)$ is \emph{strictly} log-concave.
\end{Lemma}
\begin{proof}
Our proof is inspired by Dawson's proof in \cite{dawson-1984}.
For $0\le k \le r$, we define 
 $a^k(q)= \sum_{i=0}^k a_i q^{k-i}$ and $b^k(q)=\sum_{i=0}^k b_{i,k} q^{k-i} =a^k(q+1)$.

 The polynomials $a^k(q)$ are by construction log-concave. 
 We show by induction
 over $k$ that this implies log-concavity of the polynomials $b^k(q)$. This is sufficient since $b(q)=b^r(q)$.

For $k\le 1$, nothing needs to be shown. For $k=2$, we need to check one inequality:
 \begin{align}
  b_1^2 &= (a_1 + 2a_0)^2 =  a_1^2 + 4 a_0a_1 + 4a_0^2  \\
        &\ge a_0a_2 + 4 a_0a_1 + 4a_0^2  > a_0 (a_2 + a_1 + a_0) = b_0 b_2. 
 \end{align}

 Now let $k\ge 3$. 
  Note that 
 \begin{equation*}
   b^{k+1}(q) = a^{k+1}(q+1) = (q+1)a^k(q+1)+a_{k+1} = (q+1)b^k(q) + a_{k+1}.
 \end{equation*}
This polynomial is strictly log-concave if $ (q+1)(qb^k(q) + a_{k+1}) =  q((q+1)b^k(q) + a_{k+1}) + a_{k+1}$ is, since
 setting the $q^0$ coefficient to zero followed by a division by $q$ preserves strict log-concavity.

It is an easy exercise to show that multiplication by $(q+1)$ preserves strict log-concavity of a polynomial in $q$. Hence,
it is sufficient to prove that $(qb^k(q) + a_{k+1})$ is strictly log-concave. 
By induction, we  only need to check the inequality involving the term $a_{k+1}$, \ie $b_{k,k}^2 > b_{k-1,k}a_{k+1}$:
 \begin{align}
 \label{equation:LCfirstLine}
   b_{k,k}^2 - b_{k-1,k}a_{k+1}  &= (a_0+\ldots + a_k)^2 -
                  \sum_{j=0}^{k-1}(k-j)a_j a_{k+1} \\
 \label{equation:LCsecondLine}
               & \ge (a_0+\ldots + a_k)^2  -
                  \sum_{j=0}^{k-1} \sum_{i=1}^{k-j}a_{j+i} a_{k+1-i} \\
               &= \sum_{i+j \le k} a_ia_j \ge a_0^2 \ge 1 .
 \end{align} 
To see that \eqref{equation:LCfirstLine} is greater than \eqref{equation:LCsecondLine},
 note that log-concavity of the $a_j$ implies $a_ja_{k+1} \le a_{j+i}a_{k+1-i}$ for $1 \le i \le k-j$.
\end{proof}

In a  recent preprint, June Huh proved the following result about $h$-vectors of matroids
 that was conjectured by Jeremy Dawson in \cite{dawson-1984}.

\begin{Theorem}[\cite{huh-hvector-2012}]
\label{Theorem:Huhhvector}
 The $h$-vector of a matroid complex of a matroid that is realisable 
 over a field of characteristic zero is log-concave.
\end{Theorem}

Combining this theorem with %
  Lemma~\ref{Lemma:himpliesfLC}, we obtain the following corollary
 that slightly strengthens Theorem~\ref{Theorem:fvectorLogConcave} in the case of matroids that 
 are realisable over a field of characteristic zero.
\begin{Corollary}
  The $f$-vector of the matroid complex of a matroid that is realisable 
 over a field of characteristic zero is strictly log-concave.
\end{Corollary}

\subsection{Thickenings}
\label{Subsection:Thickenings}
In this subsection we 
introduce the matroid operation $k$-fold thickening and we show
 that the $f$-vector of a ``sufficiently thick'' matroid is
 strictly log-concave if and only if its $h$-vector is.

\begin{Definition}
Let $M=(E,\Delta)$ be a matroid and let $k$ be a positive integer. 
We define the \emph{$k$-fold thickening $M^k$} of $M$ to be the matroid on the ground set $E\times\{ 1,\ldots, k\}$ 
 whose matroid complex is 
given by
\begin{align}
  \Delta^k = \{ I \subseteq E \times \{1,\ldots, k\} : \pi_E(I)\in \Delta \text{ and } \abs{\pi_E(I)}=\abs{I}  \}.
\end{align}
In this definition,
$\pi_E : E\times  \{1,\ldots, k\} \to E$ denotes the projection to $E$.
\end{Definition}
\begin{Remark}
If $M$ is realised by a list of vectors $X$, $M^k$ is realised 
 by the list $X^k$ that contains $k$ copies of every element of $X$.
\end{Remark}

\begin{Proposition}
 \label{Theorem:kThickeningHlc}
 Let $M=(E,\Delta)$ be a matroid of rank r and let $f_1$ denote the number of elements in $E$ that are not loops.
 Suppose that the $f$-vector of $M$ is strictly log-concave. 
 Then there exists an integer $k_0\le (f_1r)^{3r}$ \st
 for all  $k\ge k_0$, the $h$-vector of $M^k$, the $k$-fold thickening of $M$, is strictly log-concave.

 Put differently, for ``sufficiently thick'' matroids, the $f$-vector is strictly log-concave if and
 only if the $h$-vector is strictly log-concave. 
\end{Proposition}

\begin{Remark}
 We expect that a careful analysis will yield an upper bound on $k_0$ that is a lot stronger.
\end{Remark}

\begin{Remark}
One should note that Proposition~\ref{Theorem:kThickeningHlc} holds for arbitrary matroids and
 even for other classes of simplicial 
 complexes that have positive $h$-vectors and that 
 are closed under $k$-fold thickening.
\end{Remark}

\begin{proof}[Proof of Proposition~\ref{Theorem:kThickeningHlc}]
 First, we observe the following connection between the $f$-polynomials of $M$ and $M^k$:
\begin{align}
  \label{eq:fKthickening}
  f_{M^k}(q) = \sum_{i=0}^r k^if_i q^{r-i} = k^r f_M \left(\frac qk \right).  
\end{align}
Let $(f_0,\ldots, f_r)$ denote the $f$-vector of $M$ and let $(h_0',\ldots, h_r')$ denote the 
$h$-vector of $M^k$. By \eqref{equation:hvector},
 $h_j' = \sum_{i = 0 }^j (-1)^{j-i} \binom{ r - i }{ j-i } k^i f_i$. Hence,
\begin{align}
(h_j')^2  &= \left(\sum_{i = 0 }^j (-1)^{j-i} \binom{r- i}{j-i} k^if_i\right)^2  
  = k^{2j} f_j^2 + o(k^{2j})
  \label{equation:hjsquared}\\  
 h_{j-1}'h_{j+1}' &=  \left(\sum_{i = 0 }^{j-1} (-1)^{j-i} \binom{r- i}{j-i-1} k^if_i\right)
   \left(\sum_{i = 0 }^{j+1} (-1)^{j-i} \binom{r- i}{j-i+1} k^if_i\right)
  \nonumber
  \\ 
 &= k^{2j} f_{j-1}f_{j+1} + o(k^{2j}).
 \label{equation:hjbla}
\end{align}
 Thus, for large $k$,
 $(h_j')^2 > h_{j-1}'h_{j+1}'$ is equivalent to $f_j^2 > f_{j-1}f_{j+1}$. 
 The latter inequality holds by assumption.

For the upper bound on $k_0$,  note that Ed Swartz proved in \cite{swartz-2005} that
\begin{align}
 f_i \le \sum_{j=0}^i \binom{r-j}{r-i}\left(\binom{r-1}{j}h_r + \binom{r-1}{j-1}\right).
 \label{equation:fVectorBound}
\end{align}
$h_r$ can be bounded above by the following argument:
the $h$-vector of a matroid complex is the $h$-vector of a multicomplex \cite[Theorem II.3.3]{stanley-1996}.
It follows directly from \eqref{equation:hvector} that $h_1=f_1-r$. Hence, 
$h_r \le \binom{f_1 -1}{r-1}$. Thus, we can deduce from \eqref{equation:fVectorBound} that
 $f_i\le r^{2i}f_1^r$. %
  Comparing this with \eqref{equation:hjsquared} and \eqref{equation:hjbla} 
 implies the upper bound.
\end{proof}

\begin{Remark}
 Jason Brown and Charles Colbourn
 showed %
 that every matroid has a 
 thickening \st its $h$-polynomial has only real zeroes \cite{brown-colbourn-1994}.
 This implies that it is log-concave. 
 Here, thickening denotes an operation where additional copies of 
 some elements of the ground set are added. 
 In contrast to the $k$-fold thickening, 
 the number of additional 
 copies can be different for every element.   
\end{Remark}

\subsection{Modes of $f$-vectors}
\label{Subsection:Modes}

For a unimodal sequence $f_0,\ldots, f_r$, it is interesting to find the location of its \emph{modes}, \ie 
 the element(s) where the maximum of the sequence is attained.

\begin{Remark}
  \label{Remark:ModeLocation}
  The index of the smallest mode of the $f$-vector of a rank $r$ matroid is at 
  least $\lfloor r/2\rfloor$.
  In fact, the first half of the $f$-vector  of every matroid is strictly monotonically increasing 
 \cite[7.5.1.~Proposition]{bjoerner-1992}. %
 The minimum $\lfloor r/2\rfloor$ is attained by the uniform matroid $U_{r,r}$. 
 Some matroids have monotonically increasing 
 $f$-vectors. It follows from \eqref{eq:fKthickening} that for an arbitrary matroid $M$ 
 and sufficiently large $k$, 
 the $f$-vector of the $k$-fold thickening of $M$ is strictly monotonically increasing.
\end{Remark}

\section{Zonotopal Algebra and matroid polynomials}
\label{Section:ZonotopalAlgebra}
\emph{Zonotopal algebra} is the study of several classes of graded vector spaces of polynomials that can be associated 
with a realisation of a matroid.
The Hilbert series of these spaces are matroid invariants.
 The spaces can be described in various ways and each space has a dual counterpart with the same Hilbert series.

The theory of zonotopal algebra was developed by Olga Holtz and Amos Ron \cite{holtz-ron-2011}, extending
 various previous results \eg on polynomial spaces spanned by box splines \cite{BoxSplineBook}.  
 The two zonotopal spaces that 
are of interest to us in this paper are the 
central $\Pcal$-space  $\Pcal(X)$ and the 
internal $\Pcal$-space $\Pcal_-(X)$.
These two spaces also have a natural interpretation in the theory of box splines (cf.~\cite[Corollary 9]{lenz-todd-2013}).

Given $x\in \K^r$, we denote by $p_x$ the linear polynomial in $\K[t_1,\ldots, t_r]$ whose $t_i$ 
 coefficient is the $i$th coordinate of the vector $x$, \eg for $x=(2,1)$, $p_x=2t_1 + t_2$. 
\begin{Definition}
Let $\K$ be a field %
and let $X=(x_1,\ldots, x_N)\subseteq \K^r$ 
 be a list of vectors spanning $\K^r$. 
We define the 
\emph{central $\Pcal$-space}   $\Pcal(X)$
 and the \emph{internal $\Pcal$-space}   $\Pcal_-(X)$ by
\begin{align}
\Pcal(X) &:=
 \spa \left\{  \prod_{ x\in Y } p_x : Y\subseteq X,\, \rank(X\setminus Y)=r  \right\} \\
\text{and  }  \Pcal_-(X) &:= \bigcap_{x\in X} \Pcal(X\setminus x).
\end{align}
\end{Definition}
The Hilbert series of these two spaces are 
 evaluations of the Tutte polynomial $T_X(x,y)$ of the matroid defined by $X$ 
 \cite{ardila-postnikov-2009,berget-2010,holtz-ron-2011}:
\begin{align}
 \hilb(\Pcal(X),q) &= q^{N-r} T_X(1, \frac 1q) 
 \label{eq:CentralPTutte}
  \\
\text{and } \hilb(\Pcal_-(X),q) &= q^{N-r} T_X(0, \frac 1q).
\end{align}
Let $X^*\in \K^{(N-r)\times r}$ denote a list of vectors realising the matroid dual to the matroid realised by $X$.
In the central case, we obtain 
\begin{align}
 \label{eq:CentralHilbAsTutteEval}
 q^{r} \hilb(\Pcal(X^*),\frac 1q) &=  T_X(q, 1) 
\end{align}
by dualising and by reversing the order of the coefficients.
In the internal case, we obtain 
\begin{align}
 \label{eq:InternalHilbAsTutteEval}
 q^{r} \hilb(\Pcal_-(X^*),\frac 1q) &=  T_X(q, 0) 
\end{align}
by dualising and by reversing the order of the coefficients.
By comparing \eqref{eq:CentralHilbAsTutteEval}
 and \eqref{eq:InternalHilbAsTutteEval} with the definitions in Section~\ref{Section:MatroidAndMatroidPolynomials}
 we obtain the following result. %
\begin{Proposition}
 Let $\K$ be a field  and let $X\subseteq \K^r$ be a list of vectors spanning $\K^r$.
 Then 
\label{Proposition:TutteFcharacteristicPolynomials}
 \begin{align}
   f_X(q) &=  T_X(q+1, 1) =  (q+1)^{r} \hilb(\Pcal(X^*),\frac 1{q+1})   \\
  \text{and } (-1)^r\chi_X(-q) &= T_X(q+1,0) = (q+1)^{r} \hilb(\Pcal_-(X^*),\frac 1{q+1}).
 \end{align}
\end{Proposition}

\begin{Example}
 \label{Example:ZonotopalAlgebra}
  Let $X=((1,0),(0,1),(1,1))$. $X$ realises the uniform matroid $U_{2,3}$ and $X^*=(1,1,1)$. The Tutte polynomial is
  $T_X(x,y)=x^2 +x + y$.
\begin{align*}
 \Pcal(X^*)&=\spa\{1,t,t^2\} &
 \Pcal_-(X^*)&=\spa\{1,t \} \\
 q^2\hilb(\Pcal(X^*),1/q) &= q^2 + q + 1  &
 q^2\hilb(\Pcal_-(X^*),1/q) &= q^2 + q  \\
 f_X(q)&=q^2 + 3q+3 &
 \chi_X(-q)&=q^2 + 3q+  2 
\end{align*}
\end{Example}

\begin{Proposition}
\label{Proposition:InternalCentralFreeExtension}
 Let $\K$ be a field %
 and let $X\subseteq \K^r$ be a list of vectors spanning $\K^r$.
 Let $x\in \K^r$ be generic, \ie $x$ is not contained in any (linear) hyperplane spanned
by the vectors in $X$. Then 
\begin{align}
 \Pcal_-(X,x) = \Pcal(X) .
\end{align}
\end{Proposition}

\begin{proof}
By definition, %
 $\Pcal_-(X,x) = \bigcap_{y\in (X,x)} \Pcal((X,x)\setminus y)$.
This implies $\Pcal(X)\supseteq \Pcal_-(X,x)$.
Equality can be established by a dimension argument: in  
\cite{holtz-ron-2011}, it is shown that the dimension of $\Pcal(X)$ is equal to the number of bases that
 can be selected from $X$ and that the dimension of $\Pcal_-(X)$ equals the number of internal  bases in $X$, \ie bases that 
 have no internally active elements. It can easily be seen that $B\subseteq (X,x)$ 
 is an internal basis if and only if $B$ is a basis and $x\not\in B$.  
\end{proof}

\begin{Remark}
 Proposition~\ref{Proposition:TutteFcharacteristicPolynomials}
 and Proposition~\ref{Proposition:InternalCentralFreeExtension} imply 
 Proposition~\ref{Proposition:IndependenceCharacteristic} for realisable matroids.
 This is how the author rediscovered the connection between the characteristic polynomial and the $f$-polynomial. 
\end{Remark}

 \renewcommand{\MR}[1]{} %
\bibliographystyle{amsplain}

\providecommand{\bysame}{\leavevmode\hbox to3em{\hrulefill}\thinspace}
\providecommand{\MR}{\relax\ifhmode\unskip\space\fi MR }
\providecommand{\MRhref}[2]{%
  \href{http://www.ams.org/mathscinet-getitem?mr=#1}{#2}
}
\providecommand{\href}[2]{#2}

\end{document}